\documentclass[12pt,reqno]{amsart}

\usepackage{amsmath}
\usepackage{amsthm}
\usepackage{amscd}
\usepackage{amssymb}
\usepackage{graphicx}
\usepackage{epsf}
\theoremstyle{plain}
\newtheorem{theorem}{Theorem}[section]
\newtheorem{proposition}[theorem]{Proposition}
\newtheorem{lemma}[theorem]{Lemma}
\newtheorem{corollary}[theorem]{Corollary}
\newtheorem{claim}[theorem]{Claim}
\newtheorem{conjecture}[theorem]{Conjecture}

\newtheorem{definition}[theorem]{Definition}
\newtheorem{remark}[theorem]{Remark}
\newtheorem{example}[theorem]{Example}

\def\less{<}

\def\strutdepth{\dp\strutbox}
\def \ss{\strut\vadjust{\kern-\strutdepth \sss}}
\def \sss{\vtop to \strutdepth{
\baselineskip\strutdepth\vss\llap{$\diamondsuit\;\;$}\null}}

\def\strutdepth{\dp\strutbox}
\def \sst{\strut\vadjust{\kern-\strutdepth \ssss}}
\def \ssss{\vtop to \strutdepth{
\baselineskip\strutdepth\vss\llap{$\spadesuit\;\;$}\null}}

\def\strutdepth{\dp\strutbox}
\def \ssh{\strut\vadjust{\kern-\strutdepth \sssh}}
\def \sssh{\vtop to \strutdepth{
\baselineskip\strutdepth\vss\llap{$\heartsuit\;\;$}\null}}

\def\int{\text{int}}
\def\invlimit{\smash{\lim\limits_{\raise1pt\hbox{$\longleftarrow$}}}\vphantom{\big(}}
\def\inter{\hskip 1.5pt\raise4pt\hbox{$^\circ$}\kern -1.6ex}
\def\skel(#1,#2){#1^{(#2)}}
\def\hyp {\hbox {\rm {H \kern -2.8ex I}\kern 1.25ex}}
\def\reals {\hbox {\rm {R \kern -2.8ex I}\kern 1.15ex}}
\def\integers {\hbox {\rm { Z \kern -2.8ex Z}\kern 1.15ex}}
\def\naturals {\hbox {\rm {N \kern -2.8ex I}\kern 1.20ex}}
\def\rationals {\hbox {\rm { Q \kern -2.2ex l}\kern 1.15ex}}
\def\hyp {\hbox {\rm {H \kern -2.7ex I}\kern 1.25ex}}

\begin{document}

\title{Heegaard splittings of Twisted Torus Knots}

\author{Yoav Moriah} 
\address{Department of Mathematics\\
Technion\\
Haifa, 32000 Israel}
\email{ymoriah@tx.technion.ac.il}

\author{Eric Sedgwick }
\address{Department of Computer Science\\
DePaul University\\
Chicago, USA}
\email{esedgwick@cs.depaul.edu}

\thanks {This research is supported by  grant No.~2002039 from the US-Israel Binational  Science Foundation (BSF) Jerusalem Israel. We also would like to thank the university of DePaul in Chicago and the Technion in Haifa for their
hospitality}

\date{\today}

\begin{abstract} Little is known on the classification of Heegaard splittings for hyperbolic $3$-manifolds. Although Kobayashi gave a complete classification of Heegaard splittings for the exteriors of 
$2$-bridge knots, our knowledge of other classes is extremely limited. In particular, there are very few hyperbolic manifolds that are known to have a unique minimal genus splitting. Here we demonstrate that an infinite class of hyperbolic knot exteriors, namely exteriors of certain ``twisted torus knots" originally studied by Morimoto, Sakuma and Yokota, have a unique minimal genus Heegaard splitting of genus two.  We  also conjecture that these manifolds possess irreducible yet weakly reducible splittings of genus three. There are no known examples of such Heegaard splittings.

\end {abstract}

\maketitle

\section {Introduction}\label{introduction}

\vskip 10pt

The only class of hyperbolic manifolds for which there is a complete classification is that of the exteriors
of $2$-bridge knots. This was done by T. Kobayashi in ~\cite{Ko1}. In particular there are very few manifolds which are known to have a unique minimal genus Heegaard splitting. We are interested, for 
reasons which will become clear later, in the following class of knots:

\vskip10pt

\begin{definition} \label{defttk}\rm The knot $K \subset S^3$ obtained by taking the non-trivial 
$(p,q)$-torus knot  $K(p,q) \subset S^3$ (embedded on a standard torus $V \subset S^3$), removing a 
neighborhood of a small unknotted $S^1$ around $m$ adjacent strands, which we 
denote by $C$, and doing a  $\frac {1}{s}$ - Dehn surgery along $C$ will be called a 
{\it twisted torus knot} and  setting $r = 2s$, denoted by $T(p,q,m,r)$. Since we always will take $ m = 2$ 
strands we abbreviate to  $T(p,q,r)$. (See also ~\cite{De}  and Fig. 4 ).

\end{definition}

\vskip10pt

The main theorem we prove is the following:

\vskip10pt

 {\bf Theorem 4.9:} {\it Let  $ K_m = T(p,q,r)$ be a twisted torus knot with $(p,q) = (7,17) $ and $r = 10m - 4, m \in \mathbb{Z}$. Then for sufficiently large $m \in \mathbb{Z}$ the  knot complement 
 $S^3 - N(K_m)$ has a unique, up to isotopy, minimal genus  Heegaard splitting. }

\vskip10pt 

We suspect that the restrictions on the parameters of the knots $T(p,q,r)$ are only a result of the non-trivial  proof and thus conjecture:

\vskip10pt

{\bf Conjecture 4.16:} {\it All knot exteriors $E(K)$, where  $K = T(p,q,2,r)$ and $K$ is not $\mu$-primitive, have a unique (minimal) genus two Heegaard splitting}.

\vskip10pt

A sub-class of twisted torus knots, including those which are covered in Theorem 4.9,  was previously studied by Morimoto, Sakuma and Yokota in ~\cite{MSY} and were shown to be not $\mu$-primitive 
(see Definition \ref{primitive}).  A different sub-class was studied by J. Dean in ~\cite{De}, and  contains  knots which admit  surgeries yielding Seifert fibered spaces over $S^2$ with three exceptional fibers. This accumulation of  information points to twisted torus knots as a very interesting class of knots.

In trying to classify the higher genus Heegaard splittings for these knots one must naturally address
the issue of non-minimal genus non-stabilized Heegaard splittings. One of the more tantalizing problems in the current Heegaard theory of $3$-manifolds is that it is not known whether manifolds with fewer than two boundary components can possess non-minimal genus Heegaard splittings which are weakly reducible and non-stabilized. (In alternate terminology; non-minimal genus Heegaard splittings which are  distance one, where distance is in the sense of Hempel's distance for Heegaard splittings.) The basic problem is that there are currently no known techniques  to show that a non-minimal genus weakly reducible Heegaard splitting of a closed manifold is non-stabilized. 

In this quest to find such examples it is necessary to find a Heegaard splitting for a knot manifold that 
is not $\gamma$-{\it primitive} (see definition in Section \ref {Preliminaries} below), since boundary stabilizations of such splittings are always stabilized (see  ~\cite{MS} Theorem 4.6). In  ~\cite{MSY} Morimoto, Sakuma and Yokota showed that the knots  $K_m = T(7,17, 10m - 4)$ are not $\mu$-primitive, which is an important special case. Hence the following theorem shows that  $E(K_m) = S^3 - N(K_m)$ are candidates for manifolds which could have the desired splittings as above:

\vskip10pt

 {\bf Theorem 4.4:} {\it  The knots $ K_m = T(7,17, 10m - 4)$ are not $\gamma$-primitive for  
 \underbar {all} curves   $\gamma \subset \partial S^3 - N(K)$.}
 
 \vskip10pt

Theorems \ref{noprimitive} and the theorem of Morimoto, Sakuma and Yokota in ~\cite{MSY}  provide motivation for the following  conjecture:

\vskip10pt

{\bf Conjecture 5.1:} {\it The boundary stabilized genus three Heegard splitting $(V', W')$ of the unique
minimal genus two Heegaard splitting $(V, W)$ of $E(K_m)$, where $K_m = T(7,17,2,10m - 4)$, 
is  non-stabilized. }

\vskip10pt 

We  discuss, in Section  \ref{candidate}, some of the issues  which arise in the attempt to prove that the above boundary stabilized Heegaard splittings are indeed non-stabilized   (for some background 
see ~\cite{MS}).

\vskip20pt

\section {Preliminaries}\label {Preliminaries}

\vskip10pt

In this section we recall some definitions and notation used in the paper. We then define the 
notions of {\it primitive} and {\it weakly primitive} curves on the boundary of a knot space:

\vskip 10pt

A  {\it compression  body} $W$ of genus $g$ is a compact connected orientable $3$-manifold 
which can be  represented as $(S \times [0,1]) \cup \{ 2-handles\}$, where $S$ is a genus $g$ 
closed orientable  surface and the $2$-handles are attached to $S \times \{0\}$. We require that 
all  $S^2$  components in $(S \times \{0\}) \cup \{ 2-handles\}$ be eliminated  by attaching 
$3$-balls to them. The connected  surface $S \times \{1\}$  will be denoted by $\partial_{+}W$ 
and the not necessarily connected surface $\partial W \smallsetminus  \partial_{+}W$ will be
denoted by $\partial_{-}W$

Any  compact connected orientable $3$-manifold $M$ has a decomposition as $M = V \cup W$
and $V \cap W = \partial_{+}V  = \partial_{+ }W = \Sigma$, where $V$ and  $W$ are compression bodies. 
Such a decomposition is called a {\it Heegaard splitting} and the surface $\Sigma$ will be called
a {\it Heegaard surface}.

\vskip10pt

\begin{definition}\label{primitive} \rm \hfill
\begin{enumerate}

\item A simple closed curve on the boundary of a compression body  $W$ will be called {\it primitive} 
if it meets an essential disk $D \subset W$ in a single point. An annulus on the boundary of a compression will be called {\it primitive} if its core curve is primitive.

\item Let $K \subset S^3 $ be a knot and  $(V, W)$ a Heegaard splitting of  
$S^3 - N(K)$ with $\partial (S^3 - N(K)) \subset V$.  Let $\gamma$ denote a simple closed 
curve on $\partial_{-}V$.  We say that $(V, W)$ is $\gamma$-{\it primitive}  if  there
is a vertical annulus $A$ in the compression body $V$ such that $\partial A = \gamma \cup \beta$
where $\beta   \subset  \partial_{+}V$ meets an essential disk  $D$ of $W$ in a single point.
If $\gamma$ is a meridian for $K$  we say that it is $\mu$-{\it primitive}.  (see Fig. 1).  An annulus $A$ and a disk $D$ as above  will be called a {\it $(A,D)$ - reducing pair}.

\item If there is a vertical annulus $A\subset V$ as above and an essential disk $D \subset W$
so that $A \cap D = \emptyset$ then we say that $\gamma$ is {\it weakly primitive}.

\item We say that $K \subset S^3$ is {\it $\gamma$-primitive (weakly $\gamma$-primitive)} if 
$S^3 - N(K)$ has a minimal genus  Heegaard splitting with a primitive $\gamma$ curve  (weakly primitive $\gamma$ curve). As before if  $\gamma$ is a meridian we will say that $K$ is {it $\mu$-primitive (weakly $\mu$-primitive)}.

\end{enumerate}
\end{definition}

\begin{figure}
{\epsfxsize = 3.5 in \centerline{\epsfbox{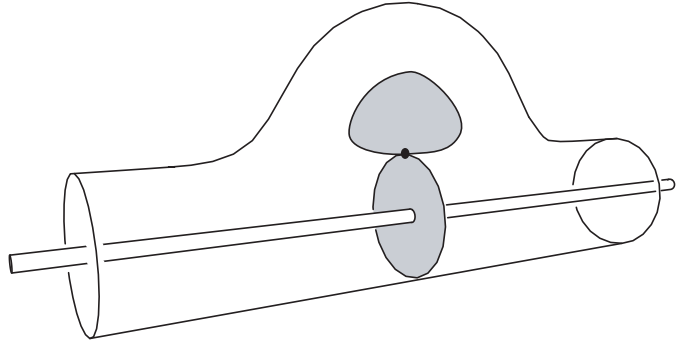}}}
\caption{A primitive meridian}
\end{figure}

\vskip10pt

\section{Heegaard splittings of Torus knots}\label{torusknots}

\vskip5pt

Since a twisted torus knot can be obtained from a torus knot by Dehn surgery in an unknot, we should
not be surprised that their Heegaard splitting are related. Hence we would like to discuss those first.
It is well know that the exterior $E(K) = S^3 - N(K) $ of a  torus knot $K = K(p,q)$ will have at most  three distinct genus two Heegaard surfaces. Generically $E(K)$ will have three distinct such  surfaces depending on the values of $p$ and $q$ (see ~\cite{Mo}).    For the purpose of this work we are interested in what we  will call the {\it middle}  Heegaard surface, defined as follows: 

\vskip0pt

\begin{definition} \label{middle} \rm
Embed the knot $K$ in the standard way into an unknotted torus  $T'$ in $S^3$. Let $T$  
denote this torus after we have removed  a single small  disk  $D$ which is disjoint from $K \subset T$.  Without loss of generality we can think of   $K$ as lying in the middle level surface  $T_0 = T \times \{0\}$ of a thickened product $T \times [-1,1]$,  . By removing a neighborhood of $K$ we obtain a compression  body $V = ( T \times [-1,1]) - N(K)$. Note that $\partial_{+}V$  is a genus two Heegaard  surface. The complementary handlebody $W$ is the union of two solid tori, $W _{i}$ the inside solid torus and 
$W_{o}$ the outside solid torus of the unknotted torus, $T'$.  These solid tori $W _{i}$ and $W_{o}$ are  attached by a 1-handle,   $D \times I$ to form  a genus two handlebody. We will say that 
 $\Sigma = \partial_+ (T \times [-1,1]) = \partial W$  is the {\it middle} Heegaard surface.

\end{definition}

\vskip0pt 

\begin{definition}\label{inandout}\rm 
The Heegaard surface obtained by taking the boundary of a regular neighborhood of the graph 
$K \cup t_{i} \cup w_{i}$ or $K \cup t_{o}\cup w_{o}$ will be called the {\it inner} and  {\it outer} 
Heegaard  surfaces and denoted by $\Sigma_i$ and $\Sigma_o$  respectively. Here $t_{i}$ is a 
small arc connecting $K$ to the core curve $ w_{i}$ of $W_{i}$ and $t_{o}$ is a small arc connecting 
$K$  to the core curve  $ w_{o}$ of $W_{o}$ (See Fig. 3).

\end{definition}

\vskip0pt

The remainder of this section will be devoted to proving the theorem below  which is of independent interest although it is not directly used in this paper. With the standard choice of $(\mu, \lambda)$ a meridian-longitude pair there is an identification between curves on the boundary of a knot space 
and a ``slope'' $r \in \mathbb{Q} \cup \{\infty\}$. When the context is clear we  often do not make the distinction between a slope and a curve.

\vskip10pt

\begin{figure}
{\epsfxsize = 3.2 in \centerline{\epsfbox{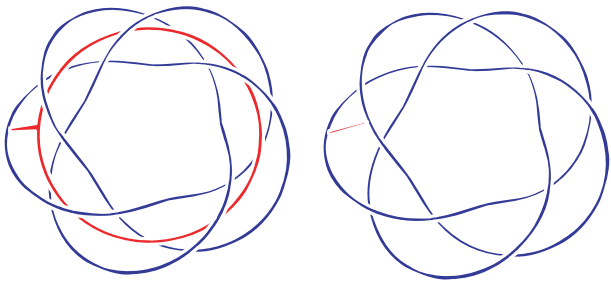}}}
\caption{Tunnel systems for a torus knot: the inner system and the middle system}

\end{figure} 

\begin {theorem} 
\label{primitive}
Let $\Sigma$ denote the  middle Heegaard surface for the non-trivial torus knot $K(p,q)$.  
The following are equivalent:
\begin{enumerate}
\item [(i)] $\Sigma$ is $\mu$-primitive, (i.e., $1/0$-primitive).

\item [(ii)] $\Sigma$ is $\gamma$-primitive for some s.c.c.  $\gamma \subset \partial (S^3 - N(K))$.

\item [(iii)]There exist $r,s \in \mathbb{Z}$ so that $|ps - rq| = 1$ and
either $r=1$ or $s=1$.

\item [(iv)] $\Sigma$ is isotopic to either the $\Sigma_i$ or $\Sigma_o$ Heegaard surfaces.
\end{enumerate}
\end{theorem}

\vskip0pt

\begin{proof}\hfill

i)  $\Longrightarrow$ ii)   Set  $\gamma = \mu$. It is a simple closed curve with slope $\frac{1}{0}$.

\vskip 5pt

iii) $\Longrightarrow$ iv)   If there exist such a pair $r,s$ then there is an $(r,s)$- curve 
on the torus $T$  which meets the knot $K(p,q)$ in a single point.  This  curve is a tunnel 
representing the middle splitting.   Furthermore since either  $r=1$ or $s=1$, when this tunnel 
is pushed into one of the solid tori, depending on whether $r=1$ or $s=1$, it represents  a core
there. Hence $\Sigma_{mid}$ is isotopic to either the $\Sigma_i$ or $\Sigma_o$ splittings.

\vskip5pt

iv)  $\Longrightarrow$ i)  Both  the $\Sigma_i$ and $\Sigma_o$ splittings are  $\mu$-primitive 
(see ~\cite{Mo}).   If the middle splitting is isotopic to either, it is also $\mu$-primitive.

 \vskip5pt

ii) $\Longrightarrow$ iii)   This is the hard case.  The argument follows:

Assume that $\Sigma$ is $\gamma$-primitive, then there is a disk-annulus pair, 
$(D_\gamma,A_\gamma)$ so that $D_\gamma \cap A_\gamma$ is a single point
and  $D_\gamma \subset W$, $A_\gamma \subset V$.

Define two compressing disks, one for each of the compression bodies $V$ and $W$.   
Consider  $D_{punct}  = D \times \{0\}$, where  $D$ is the disk  from Definition \ref{middle}.
The disk  $D_{punct} $ compresses the genus two handlebody $W$ into the upper and lower 
solid tori. Let $\alpha_{p/q}$ be an arc properly embedded in $T$ so that if we shrink $D$ to a 
point  $\alpha_{p/q}$ becomes a $(p,q)$ - curve on the torus $T'$. Let 
$D_{p/q} = \alpha_{p/q} \times [-1,1]$ be the disk obtained by taking the product of  
$\alpha_{p/q} \subset T$  with the interval $[-1,1]$.  This is a non-separating disk that 
is disjoint from the knot $K$.  Arrange $A_\gamma$ and $D_\gamma$ to intersect 
$D_{punct}$ and $D_{p/q}$ minimally, subject to the condition that they meet in a single point.

\vskip10pt

\begin{figure}
{\epsfxsize = 3.2 in \centerline{\epsfbox{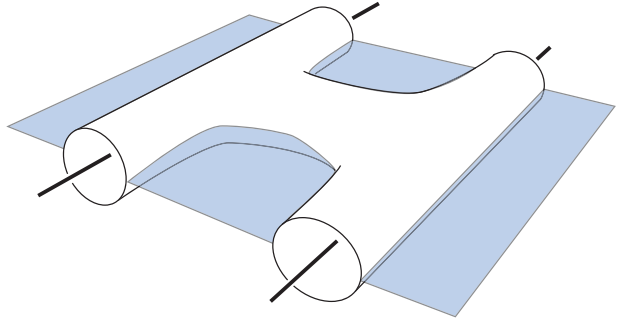}}}
\caption{The disk $D_{punct}$}

\end{figure} 

\vskip10pt

\begin{figure}
{\epsfxsize = 3.2 in \centerline{\epsfbox{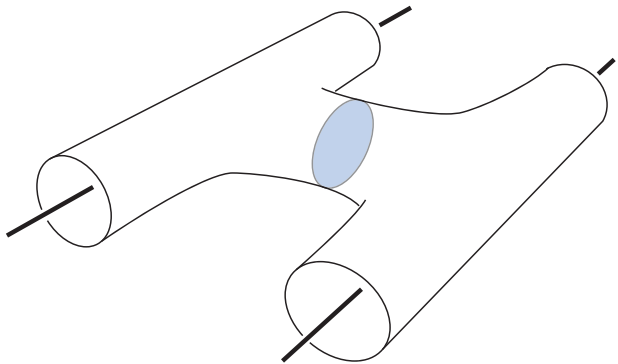}}}
\caption{The disk $D_{p/q}$}

\end{figure}

\begin{claim}

The annulus $A_\gamma$ is disjoint from the disk $D_{p/q}$.

\end{claim}

\vskip0pt

\begin{proof}

Assume in contradiction that  $A_\gamma \cap  D_{p/q} \neq \emptyset$. The intersection cannot 
contain simple closed curves which are essential on $A_\gamma$ because this would imply
that the torus knot exterior has compressible boundary. By minimality there are no inessential  such curves. So all components of intersection are arcs. Since $D_{p/q} \cap \partial X = \emptyset$ these arcs cannot be essential in $A_{\gamma}$. So $A_\gamma \cap  D_{p/q}$ is composed of inessential arcs in $A_{\gamma}$.

Consider now an outermost arc of intersection $\alpha$ on $A_{\gamma}$ it cuts off a sub-disk $E$ 
of  $A_\gamma - D_{p/q}$ and a sub-disk $E'$ on $D_{p/q}$. The disk $E \cup E'$ cannot be 
boundary parallel in the compression body by minimality. If it is isotopic to an essential non-separating 
disk then it is isotopic to $D_{p/q}$ which is the unique such disk. If the parallel region is on the same side as $E$  we can push $E$ through $D_{p/q}$ to reduce intersection. If the parallel region is on the other side, then $E$ cannot be connected to the rest of the annulus to form $A_{\gamma}$. If it isotopic to a non-separating essential disk in the genus two compression body then it is a band sum of two copies of $D_{p/q}$. However in this case, as before, $E$ cannot be connected to the rest of the annulus to form $A_{\gamma}$.

\end{proof}

\vskip10pt

Consider now the two disks  $D_{\gamma}$ and $D_{punct}$ in $W$. Assume that they  intersect  
and consider an outermost subdisk of $D_\gamma$ with respect to its intersection with  $D_{punct}$.   
We can choose such a sub-disk, called  $D_\gamma'$, that does not contain the point 
of intersection with $A_\gamma$.   This outer subdisk $D_\gamma'$ meets $D_{punct}$ 
in a single arc and since it is essential by minimality it must be  a meridional disk for 
either the upper $W_{o}$ or lower $W _{i}$ solid torus. Without loss of generality we 
will assume that it is a meridional for $W_{o}$ and represents a $0/1$ arc 
on the level torus $T_1 = T \times \{1\}$.  If $D_\gamma$ does not meet $D_{punct}$ then 
we will instead take $D_\gamma' = D_\gamma$ and recall that $\partial D_\gamma$ is a 
closed curve in $T_1$ and has a single point of intersection with $A_\gamma$. 
The argument for $W_i$ and $T_0 = T\times\{0\}$ with the slope $1/0$ is symmetric.

On the top (or bottom) punctured torus, say $T_1$, ( $T_0$) we can identify three essential arcs: 
$a_\gamma$ a component of $\partial A_\gamma \cap T_1$,  
$d_\gamma = \partial D_\gamma' \cap T_1$ and $d_{p/q} = \partial D_{p/q} \cap T_1$.
If $D_\gamma$ is disjoint from $D_{punct}$, then $d_\gamma$ is actually a closed curve. From
our preceding arguments we know that $a_\gamma$ and $d_\gamma$ are disjoint in $T_1$, as 
are $a_\gamma$ and $d_{p/q}$.    Connect the arcs across the puncture to obtain closed curves:
$\widehat{a_\gamma}, \widehat{d_\gamma}$ and $\widehat{d_{p/q}}$. Each pair of the closed 
curves obtains at most one intersection in the puncture, so we have that $\widehat{a_\gamma}$ 
and $\widehat{d_\gamma}$ meet at most once, and $\widehat{a_\gamma}$ and $\widehat{d_{p/q}}$ meet at most once. Now, the slope of $\widehat{d_\gamma}$ is  $0/1$ (or $1/0$), the slope of
$\widehat{d_{p/q}}$ is $p/q$ and  $r/s$ is the slope of the closed curve $\widehat{a_\gamma}$.

If $\widehat{a_\gamma} \cap \widehat{d_{p/q}} = \emptyset$, i.e., $\widehat{a_\gamma}$ and
$\widehat{d_{p/q}}$ are parallel and $\widehat{a_\gamma} \cap \widehat{d_{\gamma}} = \{pt\}$
or  $\widehat{a_\gamma} \cap \widehat{d_{\gamma}} = \emptyset$ i.e.,  
$\widehat{a_\gamma}$ and $\widehat{d_{\gamma}}$ are parallel   and 
$\widehat{a_\gamma} \cap \widehat{d_{p/q}} = \{pt\}$ then $|p0 - q1| = 1$ (or $|p1 - q0| = 1$) that 
is $q = 1$ (or $p = 1$) contradicting the fact that $T(p,q)$ is non-trivial. Similarly they cannot all be
parallel. If $\widehat{a_\gamma} \cap \widehat{d_{\gamma}} = \{pt\}$ and
$\widehat{a_\gamma} \cap \widehat{d_{p/q}} = \{pt\}$ then $|ps - qr| = 1$ and $|r0 - s1| = 1$, i.e, $s= 1$ (or $|r1 - s0| = 1$,  i.e., $r = 1$). Then conclusion (iii)  of Theorem \ref {primitive} holds.

\end{proof}

\vskip10pt

\section{Twisted torus knots}\label{twistedtk}

\vskip10pt

In this section we recall the definition of  a special sub-class of  twisted torus knots $T(7,17,r)$. This class will play a major role for the rest of the paper.

\begin{definition} \label{defttk1}\rm The knot $K \subset S^3$ obtained by taking the $(7,17$)-torus knot 
$T(7,17)$ in $S^3$ (embedded on a standard torus $V \subset S^3$), removing a  neighborhood 
of a small unknotted $S^1$ around $2$ adjacent strands, which we  denote by $C$,  and doing a  
$\frac {1}{5m - 2}$ - Dehn surgery, $m \in \mathbb Z$, along $C$ will be denoted by $T(7,17,2,10m - 4)$. As before, we abbreviate to  $T(7,17,10m - 4)$. (See also ~\cite{De}  and Fig. 4 ).

\end{definition}

\begin{figure}
{\epsfxsize = 4.8 in \centerline{\epsfbox{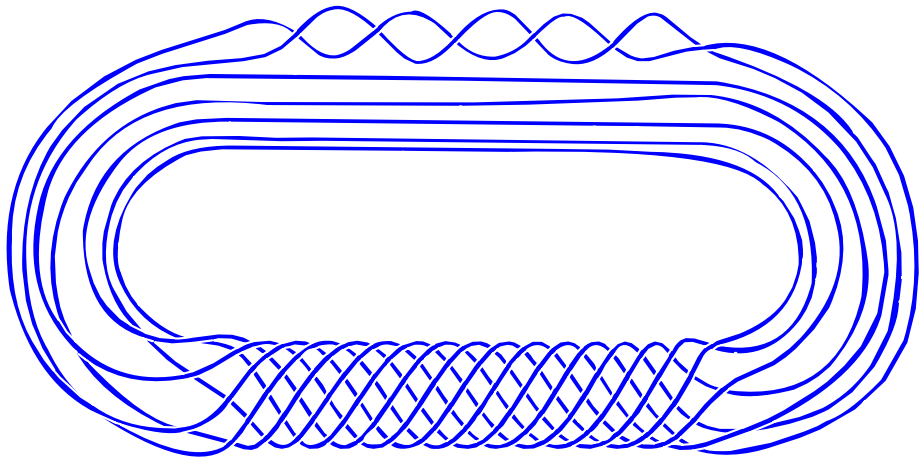}}}
\caption{A knot of Morimoto, Sakuma and Yokota }

\end{figure}

\begin{remark}\label{tunnelnumber one}\rm
All  knots in $S^3$ of the form $T(p,q,r),  r  \in \mathbb{Z}$, are tunnel 
number one knots:   Given a twisted torus knot $T(p,q,r)$ the underlying torus knot $T(p,q)$ 
has an unknotting tunnel which is an essential arc on the cabling annulus. We can consider 
such an arc in the link $L(p,q) = T(p,q) \cup C$. Since $C$ can be isotoped into the Heegaard surface
so that the slope determined by the surface on $\partial N(C)$ is $0$, then after the surgery on $C$ which gives $T(p,q,r)$,  the arc will be an unknotting tunnel for $T(p,q,r)$. Hence we have a genus two Heegaard splitting  $(V, W)$ for $E(T(p,q,r))$,   where $V$ is the compression body $\overline{N(K \cup t)}$  less  a smaller neighborhood of $T(p,q,r)$, and $W$ is the handlebody $S^3 - int(V)$.  Since this Heegaard  splitting is induced by the middle  Heegaard splitting of  the underlying torus knot $ T(p,q)$ we call it  the {\it middle} Heegaard splitting and denote it by  $\Sigma_{mid}$.

It is a theorem of Morimoto Sakuma and Yokota, (see ~\cite{MSY}),  that knots $K_m$ of the  
form $T(7,17,10m - 4), m  \in \mathbb{Z}$ are not $\mu$-primitive. Or, in other words, they do 
not have a $(1,1)$-decomposition  (see for example ~\cite{Mo1}).

\end{remark}

\vskip10pt

 \begin{proposition}\label{hyperbolic}  
 The knots $ K_m = T(7,17,10m - 4)$ are hyperbolic knots for all $m \in \mathbb{Z}$.
 
 \end{proposition}
 
 \begin{proof} \rm By Thurston if a knot is atoroidal (simple) it is either a torus knot or hyperbolic.
Non-simple tunnel number one knots have been classified by Sakuma and Morimoto  (see ~\cite{MS1}). 
In particular, all of their unknotting tunnels have  been classified (see Proposition 1.8, Theorem 2.1, 
and Theorem 4.1 of ~\cite{MS1}). They all come  from  a (1,1)-decomposition for the pair  $(S^3,K)$. 
So the knots $K$  have a primitive meridian. So do torus knots. But the knots  $T(7,17,10m - 4)$ 
do not have a primitive meridian as they are  tunnel number super additive by ~\cite{MSY}. Hence the knots $T(7,17,10m - 4)$  cannot be non-simple and thus are atoroidal. Since they are also not  torus 
knots  they must be hyperbolic.
 
 \end{proof}
 
 \vskip10pt
  
 \subsection{ The knots $K_m = T(7,17,10m - 4)$ are not $\gamma$-primitive} \hfill
 
 \vskip20pt 
 
 It turns out, as will be discussed in Section \ref{candidate} below, that in order for the knot exteriors
 of $K_m = T(7,17, 10m - 4)$ to have weakly reducible and non-stabilized non-minimal Heegaard splitting the knots  $K_m$ must have the additional property that they are not $\gamma$-primitive 
for all curves $\gamma \subset  \partial S^3 - N(K)$. This is exactly the content of our next theorem:
 
 \vskip10pt
 
 \begin{theorem} 
 \label{noprimitive}
 
 The knots $ K_m = T(7,17, 10m - 4)$  are not $\gamma$-primitive for any curve  
  $\gamma \subset \partial S^3 - N(K)$ and any $m \in \mathbb{Z}$.

\end{theorem}

\vskip10pt

\begin{proof}  By Morimoto, Sakuma and Yokota (see ~\cite{MSY}) the knots $ K_m = T(7,17, 10m - 4)$ 
are not  $\mu$-primitive. Assume that $\gamma$ is a simple closed curve on $\partial E(K)$ which is not a meridian. Since the knots $ K_m $ are tunnel number one knots their Heegaard genus $g(S^3 - N(K_m)) = 2$.  Hence,  if $K_m$ are  $\gamma$-primitive then $\gamma$-surgery on $K_m$  will give genus two manifolds  $E(K_m)(\gamma)$ with reducible genus two Heegaard splittings.  That is  $g(E(K_m)(\gamma)) \leq 1$. It follows from  ~\cite{CGLS} that since  $\gamma \neq \mu$ we cannot obtain $S^3$ by $\gamma$-surgery, hence we need  only consider the case that  $g(E(K_m)(\gamma)) = 1$. Therefore in order to prove that  $K_m$ are not  $\gamma$-primitive we need to  show that we cannot obtain lens spaces by surgery on $K_m$. 

\vskip10pt

Consider the following theorem of Ozsvath and Szabo (see ~\cite{OS}): 

\vskip10pt

\begin{theorem} 
\label{nolensspace}

If $K \subset S^3$ is a knot which admits surgery yielding a lens space, then the Alexander polynomial 
$\Delta_{K} (t) $  of $K$ has the form:

$$ \Delta_{K} (t) = (-1)^{k} + \sum_{j=1}^k  (-1)^{l-j}(t^{n_j} + t^{-n_j })$$ 

\noindent where $0 < n_1 < n_2 < ...,< n_k $ is some increasing sequence of positive integers. 

\end{theorem} 

The following theorem regarding the Alexander polynomials of twist knots $T(p,q,2n)$
denoted by $\Delta_{T(p,q,2n)}(t)$, were proved by H. Morton in ~\cite{Mt}:

\vskip 10pt

\begin{theorem}
\label{theorem1}
Suppose that $0 \less s \less \frac{q}{3}$ and $s \equiv p^{-1}$mod q. Then:

\begin{enumerate}

\item For all $n \geq 2$ the coefficient of $t^{ps+2}$ in $\Delta_{T (p,q,2n)}(t)$ is $\leq -2$.

\item For all $n \geq 2$ the coefficient of at least one of the terms $t^{ps+1}$, $t^{ps+2}$, $t^{ps+3}$
in    $t^{2n}\Delta_{T (p,q,-2n)}(t)$  is $\pm 2$.

\end{enumerate} 
\end{theorem}

\vskip 10pt

Since in our case $n = 5m-2$  we can apply the theorem for all $m \neq 0$ and conclude that
if $\{p,q\} = \{7,17\}$ then for $s = 5$ we have $ 5 \equiv 7^{-1}$ mod 17 hence the coefficient of $t^{37}$ must be at most $-2$ which violates Theorem \ref{nolensspace}.

 Though it is not in a symmetric form we can see that there are  coefficients which
 are different from $+1 , -1$ . Hence by the above Theorem \ref{nolensspace} we cannot 
 obtain a lens space by surgery and as a result $K_m$ cannot be $\gamma$-primitive.
 
\end{proof}

\vskip 10pt 

 \begin{example}\rm
 
Using Hugh Morton's  program for computing, the Alexander polynomial for 
$K_1  = T(7,17,6)$  we have:  $\Delta_{K_1}(T) $ = 

\vskip10pt
 
$  T^{102}    - T^{101}   + T^{95}   - T^{94}  + T^{88}  - T^{87}  + T^{85}  - T^{84}   
+ T^{81}  - T^{80}   + T^{78}   - T^{77}   + T^{74}   - T^{73}  + T^{71}  - T^{70}   + T^{68}  
 - T^{67}  + T^{66}  - 2 T^{65}   + 3 T^{64}   - 3 T^{63}   + 2 T^{62}   - T^{60} + T^{59}   
 - 2 T^{58}   + 3 T^{57}   - 3 T^{56}  + 2 T^{55}  - T^{53}  + T^{52}   - T^{51}  + T^{50}   
 - T^{49}  + 2 T^{47}  - 3 T^{46}  + 3 T^{45}  - 2 T^{44}   + T^{43}  - T^{42}  + 2 T^{40}
 - 3 T^{39}   + 3 T^{38}  - 2 T^{37}  + T^{36}  - T^{35}   + T^{34}  - T^{32}  + T^{31}   
 - T^{29}   + T^{28}  - T^{25}  + T^{24}  - T^{22}   + T^{21}   - T^{18}   + T^{17}   - T^{15}   
 + T^{14}  - T^{8} + T^{7}  - T + 1$
 
 \vskip 7pt
 
 In particular, the coefficient of $t^{37}$ is $-2$.
 
 \end{example}

\subsection{ The knots $T(7,17, 10m - 4)$ are not  weakly $\gamma$-primitive} 
 
 \vskip13pt 
 
 \begin{theorem} 
 \label{not weakly prim}
 The knots $K_m = T(7,17, 10m - 4)$ are not weakly $\gamma$-primitive for
 any simple curve $\gamma$ on $\partial (S^3 - N(K_m))$ and any $m \in \mathbb {Z}$.
 
 \end{theorem}

\vskip10pt

\begin{proof} Assume in contradiction that the genus two Heegaard splitting of $S^3 - (K_m)$ 
is weakly $\gamma$-primitive. Let $(A,D)$ denote the weak annulus disk pair. Compress the 
Heegaard surface $\Sigma$ along the disk $D$ and  perform surgery along the annulus $A$. 
We obtain either a torus  $T$  and an annulus $A$ or a  single  annulus $A$, depending on 
whether  the boundary of the disk separates  $\Sigma$ or not. Since the knots  $K_m$ are 
all hyperbolic their exteriors cannot contain essential annuli and tori. Hence both  $A$ and 
$T$ (if it exists) are boundary parallel. So  on the annulus $A$ we see either  one  or two 
scars from the compression along the disk. The original surface $\Sigma$ is obtained from $A$ 
by attaching a $1$-handle to $A$ either on the ``inside" or ``outside'' of $A$. If the   $1$-handle 
is on the ``outside" then there is a compressing disk $\Delta$ for the annulus $A$ on the ``inside"
which meets a meridian in a single point.  Hence $\Delta$ less a collar is a compressing disk for 
$\Sigma$ and this would imply that $K_m$ is $\gamma$-primitive in contradiction. If  the $1$-handle is on the ``inside" then $\Delta$ is a compressing disk  for the annulus $A$ meeting a meridian in a single point unless it intersects the $1$-handle in an essential way. That is  the $1$-handle is knotted. But this would mean that $\Sigma$ is not a Heegaard surface  (see Fig.  5).

\end{proof}

\vskip5pt

\begin{figure}
{\epsfxsize = 4.5 in
\centerline{\epsfbox{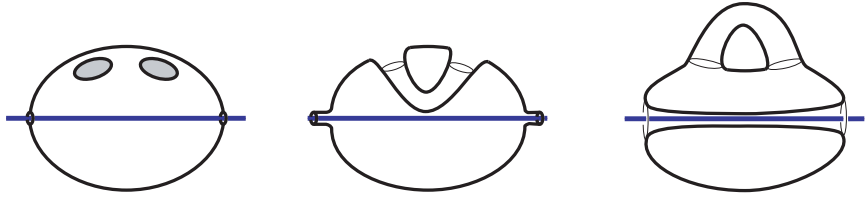}}}

\caption{A peripheral annulus with disk scars, handle inside/tube
outside, and handle outside/tube inside.}

\end{figure}

\vskip5pt

\subsection {Uniqueness of minimal genus Heegaard splittings}\label{uniquesubsection}\hfill

\vskip 10pt

Given a twisted torus knot $T(p,q,r)$  consider the link  $L(p,q) = T(p,q) \cup C$ , where $C$ is an unknotted simple closed curve encircling two strands of $T(p,q)$ as above.  The reader can check that the curve $C$ is isotopic into the middle  Heegaard surface and that the slope that $\Sigma_{mid}$ determines on $\partial N(C)$ is the $0$ slope. Hence $\frac 1 s$ - Dehn filling along $C$ is a Dehn twist on $\Sigma_{mid}$ and thus the Heegaard surface survives the twisting. On the other hand it seems on first glance that $C$ is not isotopic into the other two Heegaard surfaces for  $T(7,17)$.  Hence one might expect that the two other genus two Heegaard splittings for $E(T(p,q))$ would be ``destroyed" by the  $\frac 1 s$ - Dehn filling along $C$. In fact we prove an even stronger result:

\vskip10pt
\begin{figure}
{\epsfxsize = 4.5 in \centerline{\epsfbox{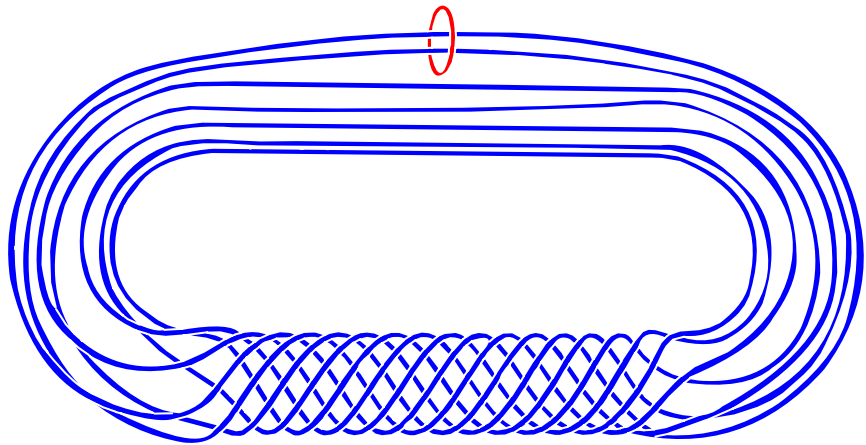}}}
\caption{The Link L(p,q)}
\end{figure}

\vskip5pt

 \begin{theorem} \label{uniqueHS} Let  $ K_m = T(p,q,r)$ be a twisted torus knot with $(p,q) = (7,17) $ and $r = 10m - 4, m \in \mathbb{Z}$. Then for sufficiently large $m \in \mathbb{Z}$ the  knot complement 
 $S^3 - N(K_m)$ has a unique, up to isotopy, genus two  Heegaard splitting. 

\end{theorem}

\vskip5pt

We first need some  lemmas:

\vskip5pt

An annulus $A \subset S^3$ will be called {\it unknotted} if the core of $A$ is unknotted as a curve in $S^3$ and the linking number of the boundary curves of $A$ is $0$. If $A^1$ and $A^2$ are unknotted annuli so that $\partial A^1 = \partial A^2$ then if the torus $ T^* = A^1 \cup A^2$ is unknotted
in $S^3$ i.e., it bounds two solid tori,  then the cores of $A^i , i = 1,2$ are  a meridian curve for one 
solid torus and a longitude curve for the other. If  $ T^* = A^1 \cup A^2$ is knotted then it bounds a solid torus on one side and a knot space on the other and the cores of $A^i , i = 1,2$ are meridian curves for the solid torus.

\vskip10pt

\begin{lemma}\label{toroidal}
The link $L(7,17) = T(7,17) \cup C \subset S^3$ is atoroidal.

\end{lemma}

\begin{proof} We first claim that $S^3 - N(L(7,17))$ is irreducible. Let $S \subset S^3 - N(L(7,17))$ 
be an essential $2$-sphere. The sphere $S$ does not separate $C$ from $T(7,17)$ since $C$ 
has linking number $2$ with $T(7,17)$.  If $S$ does not separate the two components then 
obviously it bounds a $3$-ball in the component which does not contain $C$ and $T(7,17)$: Doing
$\frac 1 n$-surgery on $C$ does not affect  either $S$ or the $3$-ball and we obtain the exterior of 
$K_m$ which is irreducible. Hence we have a contradiction to the existence of an essential $S$.

Assume  now that $S^3 - N(L(7,17))$ contains an essential torus $T$.  It follows from  
Proposition \ref{hyperbolic} that for infinitely many $\frac {1} {n}$-Dehn surgeries on  $C$,  $n = 5m - 2$, 
$ m \in \mathbb{Z}$,  we obtain a hyperbolic knot  $K_m = T(7,17,10m - 4)$  whose complement does not contain essential tori. Hence for those infinitely many surgeries either the torus $T$  compresses or becomes  a peripheral torus for  $K_m$.

If $T$ is peripheral in $ S^3 - N(K_m)$, for some $m \in \mathbb Z$, then the curves  $C$ and $K_m$ 
must be on the same side of $T$: Otherwise $T$  would  be peripheral to the $T(7,17)$ component 
of $L(7,17)$ and hence  would not be essential in $ S^3 - N(L(7,17))$.  Note that the curve $C$ and the knot  $T(7,17)$ cobound a twice punctured disk $P$. If we choose $P$ to intersect $T$ minimally the intersection cannot contain trivial curves on  $P$ as this would violate either the minimality of the intersection or the choice of $T$ as essential. Furthermore the intersection $P \cap T$  cannot be  
empty:  As then $T$ would be contained in the complement of a regular neighborhood 
of $K_m \cup P \cup C$. However  $N(K_m \cup P \cup C)$  is homeomorphic to  $N(K_m \cup t)$, where $t$ is the unknotting tunnel of  $K_m$. Thus the complement is a genus two handlebody which does not contain incompressible tori.  It cannot contain curves which are isotopic to $C$ on $P$  because in this case $C$ can be isotoped onto $T$ where, as $T$ is peripheral, $C$ is either a meridian  of $K_m$   or some other curve on $T$, in  which case $C$ is knotted. Both are contradictions. Hence  $P \cap T$  is a collection of curves each of which are concentric around one or the other of the two components  $\{p_1,  p_2\} = \partial P \smallsetminus C$.

Since $K_m$ and $C$ are on the same side of  $T$ the intersection must be an even number 
of concentric curves around each of $p_1$ and  $p_2$. If there are two or more concentric  curves around the same point consider an innermost such pair. They bound  annuli  $A_1$ on $T$ and  $A_2$ on $P$  such  that the interior of  $A_2$ lies ``outside"  $T$.  Since $T$ is peripheral the ``outside'' of $T$ is homeomorphic  to the exterior  $S^3 - K_m$ which is hyperbolic by Proposition \ref{hyperbolic}. Hence $A_2$ is boundary parallel and there is an isotopy reducing the intersection between P and T. This contradicts the minimality of the intersection $P \cap T$. Hence $P \cap T = \emptyset$ and this is a 
contradiction as above.

\vskip10pt
\begin{figure}
{\epsfxsize = 4.5 in \centerline{\epsfbox{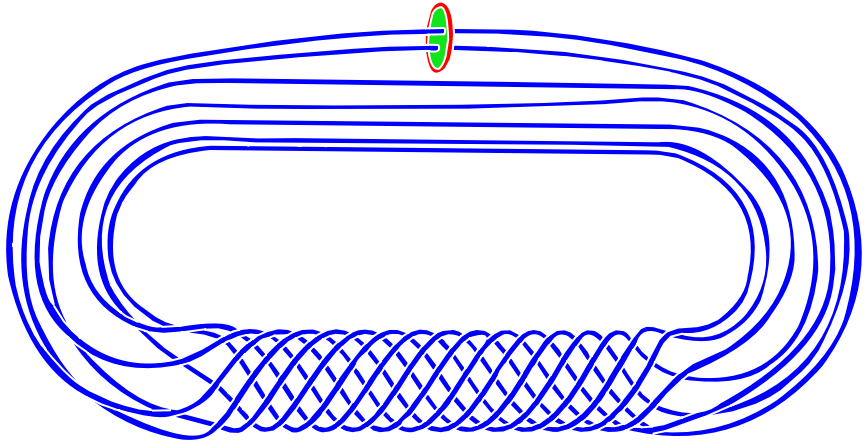}}}
\caption{The Link L(p,q) with the twice punctured disk P.}
\end{figure}

\vskip5pt

If $T$ is compressible in $S^3 - N(K_m) $ then $T$  compresses for infinitely many $\frac {1} {n}$-surgeries on $C$.  Since $S^3 - N(L(7,17))$ is irreducible then so is  $M' = S^3 - N(L(7,17)) - N(T)$.
We can now apply Theorem 2.4.4 of    ~\cite{CGLS}, which states that in this case either the 
intersection between the slopes   $\Delta(\frac {1}{10m - 4},\frac {1}{10k - 4}) \leq 1,  m \neq k$ 
which is clearly  false or that  $M_1$ the component of $M'$ which contains $\partial N(C)$ is homeomorphic to $T^2 \times I$ which is also clearly false  since then  $T$ would be peripheral to $C$. 
The last possibility is that $M_1$ is a  cable  space, i.e.,  it is homeomorphic to the complement of some  $(p,q)$-cable, $ p, q \in  \mathbb Z, g.d.c.(p,q) =1, q \geq 2$,  about the core of  a solid torus. In particular $M_1$ has two boundary components which means that  $T$ separates $C$ from $T(7,17)$.  Choose $T$ to minimize the intersection $P \cap T$. 

\vskip10pt

The intersection $P \cap T$ cannot be empty as then $T$ would not separate, and it cannot contain inessential curves on $P$ as this would violate either the minimality of the intersection or the choice of 
$T$ as essential.  

Since $P \cap T$ is minimal and $P \cap M_1 \neq \emptyset$  we can assume that $P \cap M_1$ 
is comprised only of essential annuli in $M_1$. In particular this implies that $P \cap T$ contains 
curves parallel to $C$ and hence $C$ is isotopic into $T$. Note also that $T$ must be unknotted in 
$S^3$ since it compresses to  the  $C$ side after $\frac {1}{0}$-filling on $C$ and also compresses 
to  the  $T(7,17)$ side after  $\frac {1}{0}$-filling on $T(7,17)$.  As $C$ itself is an unknot then after the isotopy onto $T$ it must be either a   $(p,1)$ or a $(1,q)$ curve with respect to $T$. That is, $C$ is 
a longitude, i.e., it meets a meridian disk for $T$ either on the $C$ side or on the $T(7,17)$
exactly once. However $M_1$ is a non-trivial cable space hence $C$ cannot meet such a meridian disk  on the $M_1$ side. Thus $C$ must meet a meridian disk for $T$ in a single point on the  $T(7,17)$ side.
However all the curves of intersection in $P \cap T$ are parallel and there is at least one (the innermost
curve) which bounds a twice punctured disk meridian disk on the $T(7,17)$ side, hence $C$ is a 
meridian there and not a longitude in contradiction. This finishes the proof and we conclude that 
 $S^3 - N(L(7,17))$ is atoroidal.

\end{proof}

\vskip10pt

\begin{lemma}\label{annular}
The link $L(7,17) = T(7,17) \cup C \subset S^3$ is an-annular.

\end{lemma}

\vskip5pt

\begin{proof} Suppose now that $S^3 - N(L(7,17))$ contains an incompressible annulus $A$.  

As above, the manifold $S^3 - N(L(7,17))$ is not  a Seifert fibered space over an annulus with a single exceptional fiber. If it was then $\frac 1{5m - 2}$-surgery on the boundary component corresponding to $C$ would yield a Seifert fibered space in contradiction to the fact that $K_m$ is hyperbolic by Theorem \ref{hyperbolic}.  

There are two possible cases, $A$ joins either distinct components of $\partial (S^3 - N(L(7,17)))$
or the same one:

If there was an annulus $A$ in $S^3 - N(L(7,17))$ with one boundary component on $T(7,17)$ and 
the  other on $C$, then a regular neighborhood $N(T(7,17) \cup A \cup C)$ has a torus boundary component $T'$ which is different from $\partial N(C)$ or $\partial N(T(7,17))$. The torus $T'$ 
contains a solid torus $V$ on the side away from $C$ and $T(7,17)$ as $S^3 - N(L(7,17))$ is 
atoroidal by Lemma \ref{toroidal}. Hence $S^3 - N(L(7,17))$ is the solid torus $V$ glued to itself along $A$  which is  a Seifert fibered space over an annulus with a single exceptional fiber, in contradiction.

Assume that  $S^3 - N(L(7,17))$ contains an annulus $A$ with both boundary components on 
 $T(7,17)$. The boundary of $N( T(7,17) \cup A)$ contains  two  tori  $T'$ and $T''$ both 
 different from  $\partial N(T(7,17))$. As $S^3 - N(L(7,17))$ is atoroidal each of  $T'$ and $T''$ 
 either bounds a   solid torus  or is  peripheral. Neither are peripheral into  $T(7,17)$ and they 
 cannot both bound  solid tori as $C$ must be somewhere. Hence one, say $T'$, is peripheral 
 into $C$. This means  that there is an annulus between $C$ and  $T(7,17)$, in contradiction. 
 The case where both boundary components of $A$ are on $C$ is identical.

\end{proof}

\vskip5pt

\begin{corollary}\label{linkhyp}
The link $L(7,17) = T(7,17) \cup C \subset S^3$ is hyperbolic.

\end{corollary}

\begin{proof}  Since $S^3 - N(L(7,17))$ is irreducible, atoroidal and an-annular it follows from Thurstons
hyperbolization theorem that it is hyperbolic.

\end{proof}

\vskip10pt

\begin{lemma}\label{unknot} Suppose $L = K \cup C$ is a two component tunnel number one link in $S^3$  with unkntting tunnel $\tau$. Assume that $C$ is the unknot in $S^3$. Then the Heegaard splitting  of $S^3 - N(K)$ induced by $\tau$ is $\mu$-primitive. 

\end{lemma}

\vskip10pt

\begin{proof}Since $C$ is unknotted  then $C \cup \tau$ is a tunnel for $K$.  Furthermore, the complement of  $N(K \cup \tau)$ is a genus  two handlebody inside the solid torus  $V = S^3 -N(C)$ so 
$(K \cup \tau)$  defines a genus two  Heegaard splitting  for $V$. This Heegaard splitting is standard 
by Casson-Gordon (see ~\cite{CG}).  The tunnel system $K \cup \tau$ is a genus two splitting of the
solid torus $S^3 - N(C)$, hence it is stabilized.   Thus the genus two compression body  
$(\partial C \times I) \cup \tau \cup K$ contains a non-separating disk $D$ which meets a disk $D'$ for the complementary handlebody in $S^3 - N(C)$ once.   But, $D$ must be the cocore of $K$, because it is the unique non-separating  disk in the compression body  $(\partial C \times I) \cup \tau \cup K$.   Together the meridional annulus $D- N(K)$ and $D'$ demonstrate that $K$ is $\mu$-primitive.

\end{proof}

\vskip0pt

We are now ready to prove the theorem:

\vskip10pt

\begin{proof} {\it (of Theorem \ref{uniqueHS})}   With the link $L = T(7,17) \cup C$ there  is an 
associated set of surgeries on $C$ which yield manifolds containing Heegaard surfaces which 
are not  Heegaard splittings for the link exterior. This set is {\it simple} as in Definition 0.5 of ~\cite{MR}: 
A subset of   $\mathbb Z \oplus \mathbb Z$ is {\it very simple} if it is a union of a finite subset 
$A \subset  \mathbb Z \oplus \mathbb Z$  and a subset of the  form $\alpha + n \beta, n \in \mathbb Z$, where $\{\alpha, \beta\}$ is some basis for  $\mathbb Z \oplus \mathbb Z$. A set is {\it simple} if it is a finite union of very simple sets.  

Let $(\mu, \lambda)$ be the  ``natural'' meridian longitude pair  for  $H_1(\partial N(C))$.
Consider the ``line"  $\mathcal{L}_0$ of surgeries containing the slopes $\frac 1 {5m - 2}$ (with respect to  $(\mu, \lambda)$). It is percisely the set of slopes  that meet the curve of  slope  $0$ once. The intersection of  $\mathcal {L}_0$ with the  simple set is contained in some ball unless $\mathcal{L}_0$ coincides with one of the lines in  the simple set. In this case it is contained in a ball union 
$\mathcal{L}_0$. Choose $m_0 \in \mathbb{Z}$  bigger in absolute value then the radius of that ball. Set  $m \in \mathbb{Z}$ such that   $|m| > m_0$.

Let $\Sigma$ be a genus two Heegaard surface which separates $S^3 - N(K_m)$ into two compression bodies  $W_1$ and $W_2$ with $\partial N(K) = \partial_- W_1$. By ~\cite{MR} Theorem 0.1 since 
 $ L = T(7,17) \cup C$ is a hyperbolic link  and $K_m$ is obtained by surgery on $C$ we can assume  
that  we can isotope $C$ into  $\Sigma$. 
 
By Theorem 5.1 and Remark 5.3 of ~\cite{RS}   the curve $C$ is a core of $ W_1$ or $W_2$ or
the surface $\Sigma - N(C)$ is either incompressible or $\Sigma - N(C)$ compresses to an  essential 
surface. Since $\Sigma$ is of genus two the latter case would imply that the essential surface  is an 
annulus. This contradicts the fact that $K_m$ is hyperbolic. In the incompressible case the slope of  
$\partial (\Sigma - N(C))$ determines some line $\mathcal{L}$ in the simple set containing the slope 
$\frac 1 {5m - 2}$. But   the slope $\frac 1 {5m - 2}$ is  in the line  $\mathcal{L}_0$. Hence by the 
choice of $m$,  the lines $\mathcal{L}$ and $\mathcal{L}_0$ coincide. Thus we conclude that the slope of   $\partial (\Sigma - N(C))$ is $0$. We now have two cases:

\begin{itemize}

\item[(i)] The curve $C$ is not isotopic to a core in either  compression body, i.e., as above 
$\Sigma - N(C$) is essential and furthermore $\partial (\Sigma - N(C))$ is of slope $0$ on $C$.

\item[(ii)] The surface $\Sigma$ is a  Heegaard surface for $S^3 - L$, i.e., $C$ is a core in (a) $W_1$ or (b) $W_2$.

\end{itemize}

\vskip5pt

\noindent (I) Assume that the curve $C$ is a core in $W_1$ (i.e, Case (ii) (a)). In this case we satisfy the conditions of Lemma \ref{unknot}. We conclude that $K_m$ is $\mu$-primitive which is a contradiction. 

\vskip10pt

\noindent (II) Assume that we are in Case (i) or Case (ii)(b). The pair of pants $P$ in  $S^3 - N(K_m )$ which is bounded by the curve $C$ also has slope $0$ with respect to $C$. In these cases the following conditions  on $C, P$ and $\Sigma$ can be satisfied:

\begin{enumerate}

\item There is an embedded annulus $A_{C}$ between $\Sigma$ and $C$ meeting $\partial N(C)$ in a curve of slope $0$.

\item Every curve in the intersection $P \cap \Sigma$ is essential in both surfaces.

\end{enumerate}

If $C$ is a core in $W_2$ then the existence of $A_C$ is obvious and condition (2) is satisfied by 
Lemma 6 of ~\cite{Sh} which guarantees that a strongly irreducible Heegaard surface can be isotoped 
to meet a properly embedded incompressible surface in essential curves in both.

In Case (i) the surfaces $\Sigma - N(C)$  and $P$ are  essential so condition (2) is automatically
satisfied. If we push $C$ slightly into $W_1$ or $W_2$ we satisfy condition (1) because $\Sigma - N(C)$
has slope $0$ on  $\partial N(C)$

Choose  $P$ and $C$ to minimize the intersection with $\Sigma$ subject to satisfying conditions (1) and (2). Thus  we can assume that  $P\cap \Sigma$ is composed  of simple closed curves and no arcs. We can assume further that the intersection  $P \cap (\Sigma - N(C))$  does not contain curves isotopic to $C$. The curve $C$ is isotopic to an innermost such curve which satisfies conditions (1) and (2) 
so  that the resulting $P$ has fewer intersections  with $\Sigma - N(C)$. We deduce therefore, that we only have simple closed curves  concentric around  $p_1$ or $p_2$ the boundary components of $P$ which are not equal to $C$.

If  there are two or more such concentric  curves around $p_1$ or  $p_2$ then there is a pair 
such that together they cobound an incompressible annulus  $A \subset P$. An innermost 
such  annulus is contained in the handlebody $W_2$ and is therefore  boundary compressible. 
Let $\partial A = \{ \alpha_1, \alpha_2\}$, with $\alpha_1$ being the interior curve on $P$, it bounds a  vertical  annulus in $W_1$ whose other boundary curve is a meridian for $K_m$.  If $A$ is boundary parallel then either it can be eliminated, thus reducing the intersection, or $C$ is contained in the solid torus determined by the boundary parallelism. In this case $C$ is parallel, using condition (1), on $\Sigma$ to  $\alpha_2$ and hence to $\alpha_1$ and  thus to a meridian of $K_m$.

Assume that $A$ is not boundary parallel (the boundary compressing disk $D$ for $A$ may meet  $C$). 
Now boundary compressing $A$ in $\Sigma$ gives an essential disk for $\Sigma$ disjoint  from 
$\alpha_1$. Hence the knot $K_m$ is weakly $\gamma$-primitive. This contradicts  Theorem \ref{not weakly prim}.

We are left with the possibility that around each of $p_1$ and $p_2$ there is at most one curve of intersection of $(\Sigma - N(C)) \cap P$. Since $\Sigma$ separates and both punctures are in 
$W_1$ there must be a single curve of intersection around each puncture.

We conclude therefore that there are exactly two  curves of intersection, $\alpha_1$ around $p_1$  and 
$\alpha_2$ around $p_2$.  This implies that there is an incompressible pair of pants $P' \subset P$ 
properly embedded in $W_2 - N(C)$ so that $\partial P' = \{\alpha_1, \alpha_2, C\}$. Choose $A_C$
that minimizes the intersection  $P' \cap A_C$. There are no arcs of intersection with end points 
on $C$ as both $A_C$ and $P'$ have slope $0$. There are no inessential arcs of intersection on $P'$
with end points on $\alpha_1$ or  $ \alpha_2$. As all the arcs of intersection are inessential on $A_C$
we can cut and paste to create another $A_C$ with fewer intersections. Hence all arcs are essential 
in $P'$.  Choose an outermost arc $\beta$ in $A_C$. Boundary compress $P'$ along that outermost sub-disk in $A_C$. If the arc $\beta$ joins $\alpha_1$ to $\alpha_2$ then the intersection of $P$ and 
$\Sigma$ is reduced by one. If $\beta$ joins $\alpha_1$ to itself then boundary compressing 
will initially increase the number of curves in $P \cap \Sigma$ but create a concentric annulus 
which can be eliminated as above. This  reduces the total number of curves of intersection.  

So we conclude that there are no concentric ($\alpha$ type) curves around $p_1$ or $p_2$.
Using the annulus $A_C$ push the curve $C$ onto $\Sigma$. That is  $P \cap \Sigma = C$. 

\vskip5pt

An application of  the following lemma finishes the proof of the theorem:

\begin{lemma} \label{SismiddleIf} If $P \cap (\Sigma - N(C)) = \emptyset$  then $\Sigma$ is isotopic to the middle Heegaard splitting $\Sigma_{mid}$ of $S^3 - N(K_m)$.

\end{lemma}

\vskip10pt

\begin{proof} The unknotted curve $C \subset S^3$ bounds a disk $\Delta \subset S^3$ which
contains the pair of pants $P$ so that $C = \partial \Delta \subset \partial P$. The disk $\Delta$
is a compressing disk for the solid torus $V = S^3 - N(C)$.  If we cut $V$ along $\Delta$ we obtain 
a $2$-tangle $(B,T)$ with two marked disks $\Delta_1$ and $\Delta_2$ each containing two points corresponding to $K_m \cap \Delta$.

Let $\Sigma$ be a genus two Heegaard splitting for $S^3 - N(K_m)$. We can assume by the above discussion that the curve $C$ is contained in $\Sigma$ as a non-separating curve. Thus cutting $V$
along $\Delta$ induces a cutting of $\Sigma$ along $C$. Hence the tangle $(B,T)$ contains the twice
punctured torus $\Sigma - N(C)$ so that 
$\partial (\Sigma -N(C)) = \partial \Delta_1 \cup \partial \Delta_2$.

The pair of pants $P$ is boundary compressible in $V$, as the only non-boundary compressible 
surfaces in a compression body are vertical annuli and disks. After boundary compressing $P$
we get two vertical annuli with boundary curves $C_1$ and $C_2$ on $\Sigma$. These vertical 
annuli are contained in respective disks $\delta_1$ and $\delta_2$ in  the component of 
$B - (\Sigma -N(C))$ which contain the strings $t_1,t_2$ of $T$. Note that  $\delta_i \cap t_i$  is a 
single point for each $i = 1,2$. The set of curves $\{C,C_1,C_2\}$ determines a pair of pants decomposition $P_1,P_2$ for  $\Sigma$

\vskip10pt

\begin{claim} \label{oneP}\rm

The pair of pants $P_i, i = 1,2,$ is isotopic in $B-N(T)$ to   $\Delta_i  - N(T), i =1,2$.

\end{claim}

\begin{proof} From the construction it follows immediately that one of the pair of pants, say $P_2$, is isotopic in $B-N(T)$ to $P$ and in particular to $\Delta_2  - N(T)$.

Consider $\Delta_1  \cup (P_1\cup \delta_1\cup\delta_2) \subset B$. It is a four times punctured 
$2$-sphere $\widehat{S} \subset (S^3,K_m)$.  The  sphere $\widehat{S}$ decomposes $K_m$ 
into two $2$-tangles $(B_1, T_1)$ and  $(B_2, T_2)$ where  $K_m = T_1 + T_2$. Since $K_m$ 
is a tunnel number one knot it follows from  ~\cite{Sc1} that  $K_m$ is doubly prime so it does 
not contain a Conway sphere and hence  $\widehat{S}$ is compressible 
in $S^3 - N(K_m)$ and in particular it is compressible in either $(B_1, T_1)$ or  $(B_2, T_2)$. 
Note that the compressing disk must separate the strings. This implies by ~\cite{Wu1} that at least 
one of the tangles, say  $(B_1, T_1)$, is a rational tangle. If the compressing disk for $\widehat{S}$ meets  $\Delta_1$ in a single arc, i.e., $(B_1, T_1)$, is an integer  tangle in the terminology of ~\cite{Wu}, we are done as then we can use the compression disk to guide the isotopy between  $(P_1\cup \delta_1\cup\delta_2)$ and $P$ . So assume it does not, i.e., the rational tangle is not an integer tangle.

Note that  $K_m$  can be decomposed into a non-trivial sum of two $2$-tangles as above if and only
the underlying torus knot $T(7,17)$ can. Now consider the cabling annulus $A \subset S^3 - N(T(7,17))$. 
When we remove $N(C)$ from $S^3 - N(T(7,17))$ to obtain $S^3 - N(L(7,17)$ $A$ is punctured twice.
Denote this twice punctured annulus by $\widehat{A}$. The intersection $P \cap \widehat{A}$ consists 
of three arcs one an essential arc of $A$ and the other two arcs run between $\partial A$ and
 $\partial N(C)$.

When we cut the solid torus $V =  S^3 - N(C)$ along $P$,  $\widehat A$ is cut as well. The result is a disk   $A' = \widehat A - P$. Note that  $A' \subset B$.  Note also that there are sub-arcs of $\partial A'$ that are on the two strands of the tangle $T \subset B$.  

If we choose $A$ that minimizes the intersection $\widehat{S} \cap A$ then the intersection 
$\widehat{S} \cap A'$ cannot contain simple closed curves. This follows since $A$ does not 
contain simple closed curves which are homologous to a sum of meridians of $T(7,17)$ and 
inessential simple closed curves on $A \cap \widehat{S}$ also bound disks on $\widehat{S}$ 
and hence can be eliminated since $B$ is irreducible.   Thus $\widehat{S} \cap A$ is a collection 
of arcs and hence $A' \smallsetminus \widehat{S}$ is a collection of   disk components. 

As the curves  $C_1$ and $C_2$ are meridional curves and $\widehat{S}$ is embedded, it follows that 
of all the arcs in $\widehat{S} \cap A'$  exactly two arcs, one on $P_1$ and one on  $\Delta_1$,  run between the two different meridional curves on $P_1$ and $\Delta_1$ respectively.  Hence one of the above disk components, say $A''$, is contained in  $(B_1,T_1)$ and runs between the strands of $T_1$. This is a contradiction as since $T_1$  is a non-integer rational tangle $(B_1,T_1)$ cannot contain such a disk $A''$. It follows that $T_1$ is an integer tangle and $P_1$ is isotopic into $P$.

\end{proof}

\vskip10pt

We claim, however, that in fact  $P_i, i = 1,2,$ is isotopic in $B-N(T)$ to $\Delta_i - N(T), i = 1,2,$
respectively. As if say, $P_1$ is isotopic to $\Delta_2$ then the two pairs of pants $P_1$ and $P_2$ 
are parallel in the rational tangle $(B_2,T_2)$, as above, which must contain $P_2$ as a sub-disk.  Hence $\Sigma$ would be a genus two surface determining a handlebody component which contains 
two meridional curves as cores. If this surface is a Heegaard surface then the fundamental group 
of  $E(T(7,17))$ can be generated by two elements represented by meridians. This contradicts 
the classification of generating systems of these groups (see ~\cite{Mo}).

Hence there is a unique way to tube to copies of $P$ so that the resulting surface is disjoint from 
$\Delta$. Thus the construction of any twice punctured torus in $(B,T)$  is unique up to isotopy in  
$B - N(T)$ and are all isotopic to the twice punctured torus since $\Sigma_{mid} - N(C)$. Thus 
$\Sigma$ is isotopic to  $\Sigma_{mid}$ and the proof of the lemma is complete.

\end{proof}

This finishes the proof of the theorem.

\end{proof}

\vskip10pt

 In order to prove the uniqueness of  the minimal Heegaard splitting in Case (I), we used the fact that 
 $K_m = T(7,17,2, 10m - 4)$ is not $\mu$-primitive.  D. Heath and H-J. Song prove in ~\cite{HS} that the 
 pretzel knot  $P(-2,3,7)$ has four non-isotopic tunnels. It is well known that it is $\mu$-primitive.  Hence 
 the following conjecture  seems plausible:

\vskip10pt

\begin{conjecture} All knot exteriors $E(K)$, where  $K = T(p,q,2,r)$ and $K$ is not $\mu$-primitive, have a unique (minimal) genus two Heegaard splitting.

\end{conjecture}

\vskip 10pt

\section{Boundary stabilization and reducibility }\label{candidate}

\vskip10pt

In this section we show how the results above can be put together to obtain candidates for 
manifolds with a non-minimal genus, weakly reducible and non-stabilized Heegaard splittings.

It is a generally accepted rule amongst those  doing research on Heegaard splittings that Heegaard
splittings of small genus are easier to handle than those of large genus. Furthermore since we are dealing with questions of {\it reducibility}  there is an advantage to dealing with Heegaard splittings 
of manifolds with boundary. Having a boundary implies that the Heegaard splitting is composed from either one or two compression bodies.  The possibilities for disks inside compression bodies are
more restricted then those for handlebodies of the same genus, hence deciding whether a reducing pair of disks exists or not might be more tractable. 

Since we are trying to prove a negative i.e., that a Heegaard splitting is not stabilized, we are forced
into a proof by contradiction. Hence the argument can be expected to follow, more or less, the following theme:

Let $M$ be a $3$-manifold of genus $g$. Assume that $M$ has  a weakly reducible Heegaard 
splitting which is stabilized of genus $g + n,  n \geq 1$. Destabilize it to obtain an irreducible  
Heegaard splitting  and somehow obtain a contradiction. If we can find a manifold  which has 
a unique minimal Heegaard splitting and a weakly reducible Heegaard splitting of genus $g + 1$ 
we would have the additional option of getting a contradiction by showing that the surface we obtain after the destabilization cannot possibly be isotopic to the unique minimal genus Heegaard surface. 

To sum up, we are looking for preferably,  a tunnel number one knot  $K \subset S^3$ so that
$E(K)$ has a genus three weakly reducible Heegaard splitting and a unique genus two Heegaard
splitting. An obvious place to look for weakly reducible Heegaard splittings is Heegaard splittings 
which are {\it amalgamated}.

Consider now the exterior $E(K_m) $ for a knot $K_m = T(7,17,10m - 4)$.  It has a unique minimal
Heegaard splitting $(V, W)$, where $V$ is the compression body, of genus two. Boundary stabilize 
$(V,W)$ by amalgamating $(V,W)$ with the standard genus two Heegaard splitting of a collar of
$\partial_{-}V$, which is just $T^2 \times I$ (see ~\cite {ST}). This operation is defined and discussed 
in detail in ~\cite{MS}, Definitions 2.2 and 2.3. We obtain a weakly reducible genus three Heegaard splitting for  $E(K_m)$. By Theorem 4.6 of ~\cite{MS} if a Heegaard splitting of a knot exterior $E(K)$ 
is $\gamma$-primitive for any curve $\gamma \subset \partial E(K)$ then the  boundary stabilized Heegaard  splitting is a stabilization. i.e., the boundary stabilized Heegaard splitting contains a 
reducing pair of disks. However,  $K_m$ is not $\mu$-primitive by ~\cite{MSY} and not $\gamma$-primitive ($\gamma \neq \mu$) by Theorem \ref {noprimitive}. Hence the obvious ways for the boundary stabilized Heegaard splitting $(V', W')$  of  $E(K_m)$ to be stabilized fail. We state:

\vskip10pt

\begin{conjecture} \label{nostabilized}
The boundary stabilized genus three Heegard splitting $(V', W')$ of the unique
minimal genus two Heegaard splitting $(V, W)$ of $E(K_m)$, where $K_m = T(7,17,2,10m - 4)$, 
is  non-stabilized.

\end{conjecture} 

\vskip10pt

\begin{remark}\rm\label{otherdirection}

If we assume in contradiction that $(V', W')$ is indeed stabilized, then the surface $\Sigma''$ 
obtained by destabilizing the Heegaard surface $\Sigma' = \partial_{+}V' = \partial_{+}W'$ is 
ambient isotopic to the Heegaard surface $\Sigma = \partial_{+}V = \partial_{+}W$. 
\end{remark}

\vskip5pt

There is an additional benefit for proving Conjecture \ref{nostabilized}:

\vskip5pt

It is a well known theorem of Casson-Gordon (see ~\cite{CG}) that if a {\it closed} irreducible orientable 
$3$-manifold has a  weakly reducible  Heegaard splitting then it is Haken. It is a natural question whether this theorem can be extended to manifolds with boundary. In ~\cite {Se} the second author  gave the first  examples of manifolds  with three or more boundary components which have weakly reducible and non-stabilized minimal genus  Heegaard splittings so that when the Heegaard surface is weakly reduced the surface obtained is non-essential. This result was improved by the authors, in  
~\cite{MS}, to manifolds with just two boundary components. It is still an open question if such an example exists for manifolds with a single boundary component.

\vskip5pt 
 
Thus, Conjecture \ref{nostabilized} would rule out the possible extension of the Casson-Gordon 
theorem ~\cite{CG} to manifolds with a single boundary component as follows:

Since $(V',W')$ is of genus three and is weakly reducible, then after weakly reducing we can obtain
either an essential $2$-sphere or an essential torus. We cannot have an essential $2$-sphere in 
a knot space as they are $K(\pi_1, 1)$'s. Since $E(K_m)$ is  hyperbolic  by Theorem \ref{hyperbolic}
any incompressible torus must be boundary parallel. This rules out the  possible extension of a 
``Casson-Gordon'' theorem to manifolds with a single boundary component.

\vskip5pt

\noindent{\bf Acknowledgments}: We would like to thank Hyun-Jong Song for pointing out that the
$P(-2,3,7)$ pretzel knot is a twisted torus knot \hfill $T(-3,5,2,\pm1)$ with four unknotting tunnels.

   \end {document}